\documentclass[a4paper,10pt]{article}
\usepackage[utf8x]{inputenc}
\usepackage{amsmath,amsthm,amssymb}

\def\MR#1{}

\makeatletter
\let\@fnsymbol\@arabic
\makeatother 

\newtheorem{thm}{Theorem}
\newtheorem{lem}[thm]{Lemma}
\newtheorem{prop}[thm]{Proposition}
\newtheorem{cor}[thm]{Corollary}
\newtheorem{claim}[thm]{Claim}

\theoremstyle{definition}
\newtheorem{definition}{Definition}

\newcommand{\FF}{\mathbb{F}}
\newcommand{\ZZ}{\mathbb{Z}}
\newcommand{\cH}{\mathcal{H}}
\newcommand{\cHLM}{\mathcal{H}^{LM}}
\newcommand{\EE}{\mathbb{E}}
\newcommand{\var}{\mathrm{Var}}
\newcommand{\Bi}{\mathrm{Bi}}

\newcommand{\formalconnected}{homologically connected}
\newcommand{\connected}{hom-connected}
\newcommand{\formalconnectivity}{homological connectivity}
\newcommand{\connectivity}{hom-connectivity}

\newcommand{\eps}{\varepsilon}

\title{Homological connectivity of random hypergraphs}
\author{Oliver Cooley\thanks{Graz University of Technology, Institute of Discrete Mathematics, 8010 Graz, Austria. email: {\tt \{cooley,kang,spruessel\}@math.tugraz.at}. Supported by Austrian Science Fund (FWF): P27290 and W1230 II}, Penny Haxell\thanks{University of Waterloo, Waterloo ON, Canada N2L 3G1. email: {\tt pehaxell@uwaterloo.ca}, Partially supported by NSERC}, Mihyun Kang\textsuperscript{1}, and Philipp Spr\"ussel\textsuperscript{1}}

\begin{document}

\maketitle

\setcounter{footnote}{2}

\begin{abstract}
We consider simplicial complexes that are generated from the binomial
random $3$-uniform hypergraph by taking the downward-closure. We determine
when this simplicial complex is homologically connected, meaning that its
zero-th and first homology groups with coefficients in $\FF_2$ vanish.
Although this is not intrinsically a monotone property, we show that it nevertheless has a single sharp threshold, and indeed prove a hitting time result relating the connectedness to the disappearance of the last minimal obstruction.
\end{abstract}

\section{Introduction}

A classical result of Erd\H{o}s and R\'enyi~\cite{ErdosRenyi59} states that
the random graph $G(n,p)$ becomes connected with high probability when $p$
is approximately $\frac{\log n}{n}$. Bollob\'as and Thomason~\cite{BollobasThomason85} subsequently proved
a hitting time result: With high probability the random graph \emph{process},
in which edges are added one at a time in random order, becomes connected at
exactly the moment when the last isolated vertex disappears. The aim of this
paper is to prove a similar result for two-dimensional simplicial complexes.

While the random graph $G(n,p)$ is defined in a canonical way -- fix the
vertex set $[n]:=\{1,2,\ldots,n\}$ and let each pair of vertices be connected
by an edge with probability $p$ independently -- there are different natural
models of random two-dimensional simplicial complexes. The following model
was introduced by Linial and Meshulam~\cite{LinialMeshulam06}: Start
with the full one-dimensional skeleton on the vertex set $[n]$ and let each
triple of vertices form a $2$-simplex with probability $p$ independently.
For such a complex $X$, they consider the first homology group $H_1(X;\FF_2)$
with coefficients in $\FF_2$ and prove that the vanishing of this homology
group has a sharp threshold at $p=\frac{2\log n}{n}$.

In this paper, we consider two-dimensional complexes that arise from the binomial
random $3$-uniform hypergraph, in which each triple forms an edge with
probability $p$ independently, by taking the downward-closure.
Either model might be considered natural. Linial and Meshulam construct their
complexes ``bottom up''; thus, all $1$-simplices have to be added to avoid
restricting the possible set of $2$-simplices. On the other hand, our complexes are
constructed ``top down'' in the sense that we first choose the $2$-simplices and then
take only those $1$-simplices needed to make the resulting structure a valid simplicial
complex. However, we keep all of the $0$-simplices (vertices) even if they are
not contained in any $1$- or $2$-simplices since they were integral to the initial
construction of the random hypergraph. (Deleting isolated vertices would leave a hypergraph
which is \emph{not} distributed as the binomial random hypergraph.)
We discuss some further models in Section~\ref{sec:conclusion}.

Unlike the complexes defined by Linial and Meshulam, a simplicial complex
generated by a $3$-uniform hypergraph does not have to be topologically
connected. Therefore, we shall call a complex $X$ \emph{\formalconnected} if \emph{both}
its first and its zero-th homology group with coefficients in $\FF_2$ vanish. This notion of connectivity
will turn out \emph{not} to be monotone increasing---adding simplices to a
\formalconnected\ complex might yield a complex that is not \formalconnected.
Nevertheless, we will show that \formalconnectivity\ has a single sharp threshold.

\subsection{Definitions and model}

A family $X$ of non-empty finite subsets of a set $V$ is called a
\emph{simplicial complex} if it is downward-closed, i.e.\ if every non-empty set
$A$ that is contained in a set $B\in X$ also lies in $X$.
The elements of $X$ of size $k+1$ are called the \emph{$k$-simplices}
of $X$. If a complex has at least one $k$-simplex but no $(k+1)$-simplices,
then we call the complex \emph{$k$-dimensional} or a \emph{$k$-complex}.
In a slight abuse of terminology, we also use $k$-complex to refer to a
simplicial complex with dimension smaller than $k$, i.e.\ with no
$k$-simplices.

Every $3$-uniform hypergraph $H=(V,E)$ generates a $2$-complex $\cH$ by
taking the downward-closure of $E$.
More particularly, we will generate the hypergraph $H$, and thus also the corresponding complex, randomly.
Let $H^3(n,p)$ denote the random $3$-uniform hypergraph on the vertex set $[n]:=\{1,2,\ldots,n\}$ in which every
triple of vertices forms an edge with probability $p$ independently. This is the random \emph{binomial} model,
but we will also need to consider the random \emph{uniform} model $H^3(n,m)$, the random $3$-uniform
hypergraph on vertex set $[n]$ which has its edge set $E$ chosen uniformly at random from $\binom{\binom{[n]}{3}}{m}$.
We denote the corresponding random $2$-complexes by $\cH_p(n)$ and $\cH_m(n)$, respectively.
When the number of vertices is clear from the context, we will often omit $n$ in these notations.
By $\cHLM_p$ and $\cHLM_m$ we denote the random $2$-complex obtained from
$\cH_p$ and $\cH_m$ respectively by adding all $1$-simplices in $\binom{[n]}{2}$.
Thus, $\cHLM_p$ is the random $2$-complex that was considered by
Linial and Meshulam~\cite{LinialMeshulam06}.

We will usually consider the
complex instead of the hypergraph and refer to its 2-simplices as
\emph{faces}, to its 1-simplices as \emph{edges}, and to its
0-simplices as \emph{vertices}. The set of vertices and edges forms the
\emph{shadow graph}.

\begin{definition}
  A simplicial complex $X$ is called \emph{\formalconnected}, abbreviated to
  \emph{\connected}, if its zero-th and first homology groups with coefficients in $\FF_2$, denoted by $H_0(X;\FF_2)$ and
  $H_1(X;\FF_2)$, vanish. The zero-th homology group vanishing is equivalent
  to being topologically connected, i.e.\ the shadow graph is connected. By the
  equality of simplicial homology and cohomology, the first homology group
  vanishes if and only if the first cohomology group does, i.e.\ if every
  $1$-cocycle is a $1$-coboundary, which can be stated in the following way.
  
  \medskip
  \noindent
  \begin{minipage}{\linewidth}
    \hfill
    \textbf{(H1)}
    \hfill
    \begin{minipage}[c]{0.85\linewidth}
      For every 0-1 function $f_e$ on the edges of $X$ that has an even
      number of 1s on the boundary of each face of $X$ there is a 0-1
      function $f_v$ on the vertices of $X$ such that $f_e$ is 1 for
      precisely those edges whose end vertices have different values
      for $f_v$.
    \end{minipage}
    \hfill{}
  \end{minipage}
  \medskip

  We call a 0-1 function on the edges of $X$ \emph{bad} if it
  contradicts property (H1), i.e., if it is even on the boundary
  of every face but is not induced by a 0-1 function on the vertices.
  This is the case if and only if the shadow graph has a cycle whose
  edge-values sum to an odd number.

The \emph{support} of a 0-1 function is the set of edges mapped to $1$.
\end{definition}

Topological connectivity in a complex generated from a $3$-uniform hypergraph
is equivalent to vertex-connectivity of the
hypergraph. The requirement for the complex to be topologically
connected did not appear in~\cite{LinialMeshulam06} since in $\cHLM_p$ all
edges are automatically present, so topological connectivity is trivial.

We further note that in contrast to the model of Linial and Meshulam, in our model
the property of being \connected\ is \emph{not} a monotone
increasing property---this is because the shadow graph is not
automatically complete, and by adding a new face we may create a
new cycle which is part of a bad 0-1 function.

It is therefore not obvious that \connectivity\ should have a threshold in $\cH_p$
or indeed that it does not exhibit several thresholds where it oscillates
between being connected and disconnected. However, our results
in this paper prove that there is in fact a single threshold.

\subsection{Main results}

Linial and Meshulam~\cite{LinialMeshulam06} proved that \connectivity\ of the random complex $\cHLM_p$ undergoes
a phase transition at $p= \frac{2\log n}{n}$; our first
main result proves the analogous result in $\cH_p$.
We will consider the asymptotic properties of
$\cH_p(n)$ as $n$ tends to infinity, hence any unspecified asymptotics
in the paper are with respect to $n$.
We say that an event holds \emph{with high probability}, abbreviated to
\emph{whp}, if it holds with probability tending to $1$ as $n$ tends to infinity.

\begin{thm}\label{thm:binomialthreshold}
Let $\omega$ be any function of $n$ which tends to infinity as $n$ tends to infinity.
Then with high probability
\begin{itemize}
\item $\cH_p(n)$ is not \formalconnected\ if $p=\frac{\log n+\frac12 \log \log n-\omega}{n}$;
\item $\cH_p(n)$ is \formalconnected\ if $p=\frac{\log n+\frac12 \log \log n+\omega}{n}$.
\end{itemize}
\end{thm}

Compared with $\cHLM_p$, the probability threshold
at which the phase transition occurs differs by approximately a factor of
$2$. The reason for this difference is that the minimal obstruction in
$\cHLM_p$ is an edge which does not lie in any face, which by definition does
not exist in our model, so our minimal obstruction is different.

Indeed, we prove a hitting time result; the process becomes connected
at the moment when the last minimal obstruction disappears. In this case, the minimal obstruction, denoted $M$, is defined as follows.

\begin{definition}
A copy of $M$ in a $2$-complex $\cH$ is a face with vertices $a,b,c$
in which the edges $ab$ and $ac$ are in no other faces, and in which there is a path
$P_{ab}$ of edges between $a$ and $b$ which does not use the edges $ab$ or $ac$.
\end{definition}

In this case a bad function $f$ would take the value $1$ on $ab$ and $ac$ and $0$ everywhere else.
Since $ab$ and $ac$ are in no further faces, every face is even.
However, $P_{ab}$ together with the edge $ab$ would form a cycle with
precisely one value $1$, ensuring that $f$ cannot be generated by a
vertex bipartition. 

We write $M\subset \cH$ if $\cH$ contains such a structure. We say that a certain
face forms an $M$ if it can be chosen as the face $abc$ in an $M$.
(Note that this is a slight abuse of terminology, since it also requires
the existence and the non-existence of some other faces.)

For a hitting time result,
we are concerned with the random hypergraph \emph{process} in
which at each time step we add a face chosen uniformly at random
from among those faces not already present. (And we then consider
the complex generated by the hypergraph at each time step.)
At time $m$, this gives us the random uniform hypergraph
$H^3(n,m)$. However, many calculations are easier in the random
binomial hypergraph $H^3(n,p)$, where $p=m/\binom{n}{3}$.

Recall that $\cH_m(n)$ denotes the complex generated by $H^3(n,m)$ and write
$(\cH_m(n))$ for the corresponding process.

\begin{thm}\label{thm:hittingtime}
With high probability, $(\cH_m(n))$ is not \formalconnected\ until the moment when the last
copy of $M$ disappears, and is \formalconnected\ thereafter.
The change happens at time
$$
m=\frac{n^2}{6}\left(\log n+\frac12 \log \log n + O(1)\right).
$$
\end{thm}

\subsection{Paper overview and proof methods}

This paper is laid out as follows.

In Section~\ref{sec:minimal} we will determine when the last minimal
obstruction disappears. Note that the presence of copies of $M$ is not
monotone, though we will show that at around the threshold determined
in the introduction, whp the complex becomes $M$-free and no copy of
$M$ will appear for the rest of the process.

In Section~\ref{sec:subcritical} we will prove the subcritical cases of Theorems~\ref{thm:binomialthreshold} and~\ref{thm:hittingtime}. The strategy is to divide the subcritical range into five subintervals
\begin{itemize}
\item $I_0:=[0,p_T)$
\item $I_1:=[p_T,p_1)$
\item $I_2:=[p_1,p_2)$
\item $I_3:=[p_2,p_3)$
\item $I_4:=[p_3,p_M)$.
\end{itemize}
Here $p_T$ is the birth time of the face which causes the complex to become topologically connected and
$p_M$ is the birth time of the face which causes the last copy of $M$ to disappear, which we prove to be
$$p_T = (2+o(1))\frac{\log n}{n^2}$$
and
$$p_M=\frac{\log n + \frac12 \log \log n +O(1)}{n}$$
whp (see Corollaries~\ref{cor:topologicallyconn} and~\ref{cor:lastobstruction}, respectively), while
$0<p_T<p_1<p_2<p_3<p_M$ are chosen appropriately. Clearly in the interval $[0,p_T)$ the complex is
not topologically connected and therefore not \connected. We then prove that whp there are four copies
of $M$, called $M_i$ for $1\le i \le 4$, such that $M_i$ exists in the complex throughout the interval
$I_i$ (Lemmas~\ref{lem:firstobstruction} and~\ref{lem:obstructionintervals}). Together, these intervals cover the
entire range $[0,p_M)$ and resolve the subcritical case.

We note that while it may be possible to reduce the number of intervals
by reducing the number of copies of $M$ we use to cover the range from
$p_T$ to $p_M$, we certainly cannot expect just one copy of $M$
to suffice. This is because to push the argument all the way to $p_M$
we must certainly pick the very last copy of $M$ to disappear. Having
no choice about which copy of $M$ we can choose means we cannot expect
it to have existed for a very long time.

The supercritical case of Theorems~\ref{thm:binomialthreshold} and~\ref{thm:hittingtime} will be proved
in Section~\ref{sec:supercritical}. Since whp the complex is both topologically connected and contains no $M$
in this range, it remains to prove that whp there are no larger obstacles to \connectivity, i.e.\ bad 0-1 functions
with support of size at least $3$. We consider a minimal bad support and observe important properties that
the minimality guarantees (Lemma~\ref{lem:minimalbad}). Weaker versions of these properties were already
considered by Linial and Meshulam. For example, a minimal bad support has to be connected. However,
for our proof to work we need an additional property which we call \emph{super-connectedness}
(Definition~\ref{def:superconnected}), which also guarantees the existence of certain faces related to this support.

With this new definition, we prove the supercritical case in two ranges of the size $k$ of the support.
First for $3 \le k \le \log n$, we have a simple application of the first moment method (Lemma~\ref{lem:smallsupport}). We then prove
the case when $k\ge \log n$, for which we need to bound the total number of possible bad supports
more cleverly than in the case of small $k$, which we do via a breadth-first search process (Lemma~\ref{lem:largesupport}).
This search process, which is the second point in the supercritical case
where our proof differs from that of~\cite{LinialMeshulam06}, allows us to track the
construction of a super-connected support and thus count the number of possibilities
much more precisely than other methods allow us to.

\section{Preliminaries}

\subsection{Intuition: Where ``should'' the threshold be?}

We first justify our definition above of the minimal obstruction $M$. Of course,
in one sense the smallest obstruction is an isolated vertex, but this is
rather an obstruction to \emph{topological} connectivity, and we expect the
complex to become topologically connected well before it becomes
\connected. We therefore assume topological connectivity
and consider what the minimal obstruction to \connectivity\ should be.

We need a 0-1 function on edges with no odd faces, and since each edge lies
in at least one face, this automatically means the support must have size
at least two.\footnote{Note that this is the reason why the \connectivity\ threshold in this
model is different to the threshold for the Linial \& Meshulam model in~\cite{LinialMeshulam06},
in which the minimal obstruction is an edge which does not lie in a face---in our
model, by definition, such an edge does not exist.}
The structure $M$ described above gives rise to just such a function,
as previously described.

Note that the property $M\subset \cH$ is also \emph{not} a monotone property---it
demands the existence of a face and a path but also the non-existence of other faces.
This will make various arguments slightly more tricky. However, intuitively the 
hypergraph process will initially not be \connected\ and copies of $M$ will appear
before it becomes \connected. As more faces are added, the copies of $M$ will
become the only obstructions to \connectivity\ and eventually when the last
copy disappears, the hypergraph becomes \connected\ and remains so. We will prove that
this intuition is correct in the course of the paper.

Let us provide a rough argument for why the threshold for the disappearance
of the last copy of $M$ should be at about $p=\frac{\log n+\frac12 \log \log n}{n}$.

First consider when the path $P_{ab}$ is likely to appear. The probability that a
fixed edge exists is approximately $1-(1-p)^n \simeq np$ (if this is small). There are
$\Theta(n^{k-1})$ possible paths of length $k$, and so the expected number of
these is $\Theta (n^{k-1}(np)^k)$. This is constant when $p= n^{-\frac{2k-1}{k}}$,
so we can expect a constant length path to exist if $p \ge n^{-2+\delta}$, for some
small constant $\delta>0$. Note that this bound is significantly smaller than the
$p=\Theta(\frac{\log n}{n})$ that we will predominantly be considering.

Next, consider when a face with two edges contained in no other face exists. The probability of three arbitrary vertices forming such a face is approximately $3p(1-p)^{2n}\simeq 3pe^{-2pn}$, so the expected number of these is of order $n^3pe^{-2pn}$.

To determine the threshold (asymptotically approximately), we seek $p$ such that $n^3pe^{-2pn}=1$. This holds when
$$
3\log n + \log p -2pn = 0
$$
which implies
\begin{align*}
p & = \frac{3\log n+\log p}{2n}\\
& = \frac{3\log n + \log \left(\frac{3\log n+\log p}{2}\right)-\log n}{2n}\\
& = \frac{\log n+\frac{1}{2}\log \log n + O(1)}{n}
\end{align*}
and so we expect a phase transition around $p=\frac{\log n+\frac12 \log \log n}{n}$.

\subsection{Basic facts and notation}

We will often use the following standard result.

\begin{lem}[Chernoff Bound, see e.g.~\cite{LinialMeshulam06}]\label{lem:chernoff}
Given a binomially distributed random variable $X$ with mean $\mu$ and a real number $a> 0$,
\begin{align*}
\Pr(X\ge \mu + a) &\le \exp \left( - \frac{a^2}{2(\mu+a/3)}\right); \\
\Pr(X\le \mu - a) &\le \exp \left( - \frac{a^2}{2\mu}\right).
\end{align*}
\end{lem}

To aid in the transition between the two models $H^3(n,p)$ and $H^3(n,m)$ of random hypergraphs
(and thus also between the corresponding models $\cH_p$ and $\cH_m$ of random complexes),
we utilise the standard trick of birth times: For each triple of vertices,
choose a number from $[0,1]$ uniformly at random and independently
for each triple. This will be the \emph{birth time} of the corresponding
face. Then for any probability $p$, the hypergraph consisting of those
faces with birth time at most $p$ is distributed as $H^3(n,p)$, while
the hypergraph process $(H^3(n,m))$ can be obtained by ordering
the faces by increasing birth time (with probability $1$ no two faces
have the same birth time).

With this point of view, we sometimes think of $H^3(n,p)$
(and correspondingly $\cH_p$) as also being
a process in which $p$ is gradually increased from $0$ to $1$. We
sometimes talk of taking a ``union bound over $p$'' in a certain range---this
makes little sense if we think of $p$ as being able to take any
value within the interval, but if we condition on the set of birth times
then in fact we only consider $p$ taking the value of all birth times
within the appropriate interval, which is a discrete set.

Finally in order to transfer various results between models, we observe the following,
which is a simple application of the Chernoff bound.
\begin{claim}\label{claim:faces}
Given any interval $[q_1,q_2]\subset [0,1]$ (where $q_1,q_2$ may depend on $n$)
if $(q_2-q_1)\binom{n}{3} \to \infty$, then with high probability the number of birth
times within $[q_1,q_2]$ is $(1\pm o(1))(q_2-q_1)\binom{n}{3}$.
\end{claim}
\begin{proof}
Let $X$ be the number of birth times within $[q_1,q_2]$, which is distributed $\Bi(\binom{n}{3},q_2-q_1)$. Let $\mu:=(q_2-q_1)\binom{n}{3}\to \infty$.
Observe that by Lemma~\ref{lem:chernoff},
\begin{equation*}
\Pr \left( |X-\mu| \ge \mu^{2/3} \right) \le
2\exp \left( -\frac{\mu^{4/3}}{3\mu} \right) \le 2 \exp \left( -\mu^{1/4} \right) = o(1).\qedhere
\end{equation*}
\end{proof}
We will apply this claim a bounded number of times without explicitly mentioning it,
often with $q_1=0$. Since we will only apply it a bounded number of times
we may use a union bound over all error probabilities of size $o(1)$ to ensure that
the stated events still hold whp.

We also note that conditioned on an edge not being present at time $p=q_1$,
the probability that it is present at time $q_2$ is $\frac{q_2-q_1}{1-q_1}$.
Thus we may obtain $\cH_{q_2}$ from $\cH_{q_1}$ by sprinkling an additional
probability of $\frac{q_2-q_1}{1-q_1}$. Since we will only ever want to consider
such a situation with $q_1=o(1)$, we often simply take $q_2-q_1$ as an approximation
for $\frac{q_2-q_1}{1-q_1}$. This will be valid since a lower bound on the sprinkling
probability will be sufficient.

We ignore floors and ceilings when this does not significantly affect the argument.

\section{Minimal obstructions}\label{sec:minimal}

In this section we prove various results related to when copies of $M$ exist in $\cH_p$. In particular, let
$p_M^*$ be the first birth time larger than $\frac{\log n + \frac14 \log \log n}{n}$ such that $\cH_p$
contains no copy of $M$, and recall that $p_M$ is the time at which the last copy of $M$ disappears.

Note that in theory we could have $p_M^*=\frac{\log n + \frac14 \log \log n}{n}$ if at this time there are no copies of $M$, though the results of this section show that whp this does not happen.
Our goal is to show that in fact whp
$$p_M^*=p_M=\frac{\log n+\frac12 \log \log n+O(1)}{n}.$$

For the rest of the paper, let us fix some constant
$0<\eps <\frac{1}{10}$.
(We think of $\eps$ as being arbitrarily small, but any constant in
this range will be sufficient.) We will need the following basic
fact---it tells us that by the time we have probability $p=n^{-1-\eps}$,
the shadow graph is highly connected.

\begin{lem}\label{lem:denseshadow}
Let $p=n^{-1-\eps}$. Then with high probability, every pair of vertices is connected by at least
$\sqrt{n}$ paths of length $2$ in the shadow graph of $\cH_p$.
\end{lem}

\begin{proof}
Fix two vertices $x$ and $y$ and consider the number of paths of length $2$
connecting them. To ensure independence in various calculations, we will
only count paths of a certain type, which gives us a lower bound on the
total number of paths. To this end, we pick disjoint vertex sets $U$ and
$Z$ not containing $x$ or $y$ and of size $n/3$. We will count paths $xzy$
where $z \in Z$ and the edges $xz$ and $zy$ exist because there are faces
$uxz$ and $vyz$ with $u,v \in U$. Note that for fixed $x$ and $y$, all
faces which we consider are distinct, ensuring independence. Let $X$ be
the number of such paths $xzy$.

Now, the probability that a vertex $z$ is the midpoint of such a path is
equal to the probability that there are $u,v\in U$, not necessarily distinct,
with $uxz$ and $vyz$ both being faces of the hypergraph. This probability is
$$
(1-(1-p)^{n/3})^2 \ge \left(pn/3 -(pn/3)^2/2\right)^2 \ge n^{-3\eps}.
$$
This probability is independent for each $z$, so the number of paths we obtain dominates Bi$(n/3,n^{-3\eps})$, and by Lemma~\ref{lem:chernoff} the probability that this is less than $\sqrt{n}$ is at most
$$
\exp\left(-\frac{(n^{1-3\eps}/3-\sqrt{n})^2}{2n^{1-3\eps}/3}\right)\le \exp \left(- \frac{n^{1-3\eps}}{7} \right) \stackrel{\left(\eps<\frac16\right)}{\le} e^{-\sqrt{n}}.
$$
Thus we may take a union bound over all $\binom{n}{2}$ possible choices for $x$ and $y$ and the probability that the statement in the lemma does not hold is at most
$$
\binom{n}{2}e^{-\sqrt{n}} \le e^{-n^{1/3}} = o(1)
$$
as required.
\end{proof}

Lemma~\ref{lem:denseshadow} tells us that the paths necessary for an $M$
are very likely to exist. This motivates the following definition, which
is a relaxation of $M$: Let $M'$ consist of a face with two edges contained
in no other face (i.e. a copy of $M$, but without the requirement of having
an additional path in the shadow graph). Clearly if $M' \not\subset \cH$, then
also $M\not\subset \cH$. We will usually consider $M'$ in a range of $p$ where
the existence of paths is extremely likely, so the existence of $M'$ and $M$
are essentially equivalent events (though we will only ever use the bound in
the correct direction).

We next prove that in the range shortly before the critical threshold for
\connectivity, the expected number of obstructions is concentrated around
its mean. (This result is stronger than we need for this section, but the
stronger version will be necessary later on.)
To help with this we talk of \emph{rooted triples} forming a copy of $M'$.
A rooted triple is a triple of vertices $x,y,z$ in which one of these
vertices (say $x$) is the root. We say that this rooted triple forms a
copy of $M'$ if the triple forms a face and $xy$ and $xz$ are in no other faces.

\begin{lem}\label{lem:secondmoment}
Let $\omega$ be any function of $n$ which tends to infinity, let
$\frac1{n\log n}\le p \le \frac{\log n + \frac12 \log \log n - \omega}{n}$
and let $X$ be the number of rooted copies
of $M'$ in $\cH_p$. Then with high probability, $X\sim \EE(X) \ge \frac{n^3p}{3}\exp (-2pn)$.
\end{lem}

\begin{proof}
We assume without loss of generality that $\omega = o(\log \log n)$.
We start by approximating the first moment. We have
\begin{equation}
\EE(X) = \binom{n}{3}3p(1-p)^{2(n-2)} \ge (1+o(1)) \frac{n^3p}{2} \exp(-(p+p^2)2n)
\end{equation}
and in particular the desired lower bound follows since $p^2n = o(1)$. Let us note here that
the expectation is maximised at $p=\frac{1}{2n-3}$, and so minimised at either ends of the range of $p$. It is simple to check that the upper extreme of $p$ gives the smaller expectation, which we bound by
\begin{align*}
\frac{n^3p}{3}\exp (-2pn) & \ge (1+o(1))\frac{n^2\log n}{3} \exp \left( -2\log n - \log \log n +2\omega \right)\\
& = \frac{1+o(1)}{3}e^{2\omega} \to \infty.
\end{align*}

We also need to calculate the second moment. To do this we calculate the probability that two rooted triples of vertices both form minimal obstructions, distinguishing across the size of their intersection, showing that the probabilities are of similar order regardless of the intersection. Since almost all pairs of triples do not intersect, this will show that $\EE(X^2)$ is dominated by the non-intersecting pairs of triples, as required.

The contribution to $\EE(X^2)$ made by rooted triples which are the same except possibly the root is at most $9\EE(X) = o(\EE(X)^2)$. If the two triples are not the same, then we certainly require 2 faces. We claim that we also require $4n-O(1)$ non-faces in all cases. This is certainly clear in the cases when the intersection has size at most $1$, since each of four edges must lie in $n-O(1)$ non-faces, and we can only double-count faces containing two of these, of which there are at most $\binom{4}{2}$. On the other hand, if the intersection has size $2$, i.e.\ the two triples intersect in an edge, then this edge must certainly be in these two faces, and the remaining four edges are the ones which are in no further faces. Thereafter, we argue as before.

Now we note that $(1-p)^{O(1)} = 1-o(1)$, and so the probability of two rooted triples (in which the triples are not identical)
forming two copies of $M'$ is approximately the same regardless of their intersection, and in particular approximately asymptotically
the square of the probability of one rooted triple forming a copy of $M'$.

Thus the expected number of pairs of copies of $M'$ is
$$
\left(9\binom{n}{3}^2-O\left(n^5\right)\right)(1+o(1))p^2(1-p)^{4n-O(1)} = (1+o(1))\EE(X)^2.
$$
Thus $\var (X) = \EE(X^2)-\EE(X)^2 = o(\EE(X)^2)$ and by Chebyshev's inequality we have $X\sim \EE(X)$ whp, as required.
\end{proof}

\begin{lem}\label{lem:nomoreminimal}
Whp for all $p\ge p_M^*$, $\cH_p$ contains no copy of $M'$, and therefore also no copy of $M$, i.e. $p_M=p_M^*$.
\end{lem}

\begin{proof}
We begin by observing that, conditioned on the high probability event of Lemma~\ref{lem:denseshadow},
a copy of $M'$ can only appear if there are two incident pairs which are both not in $\cH_{p_M}$, but such
that the triple containing both of them is born as a face before
any triple containing either one or the other.

We therefore first bound the number of pairs of incident pairs not in $\cH_p$, for $p=\frac{\log n + \frac14 \log \log n}{n}$. The probability that two incident pairs are both not in $\cH_p$ is $(1-p)^{2n-3} \le (1+o(1))e^{-2pn} = O\left(\frac{1}{n^2\sqrt{\log n}}\right)$. Therefore the expected number of such pairs is $O\left(\frac{n}{\sqrt{\log n}}\right)$ and by Markov's inequality, whp there are at most $\frac{n}{\sqrt[3]{\log n}}$ of them.

Given such a pair, the probability that the face containing both of them is born before any face containing just one is of order $1/n$. Therefore the expected number of times a copy of $M'$ is created throughout the rest of the process is $O(1/\sqrt[3]{\log n})$, and so
whp none are created, as required.
\end{proof}

\begin{cor}\label{cor:lastobstruction}
With high probability, $p_M= \frac{\log n +\frac12\log \log n +O(1)}{n}$.
\end{cor}

\begin{proof}
Whp we have $p_M\ge \frac{\log n + \frac12 \log \log n -\omega}{n}$ for any $\omega \to \infty$ by Lemma~\ref{lem:denseshadow} and Lemma~\ref{lem:secondmoment}. On the other hand if $p= \frac{\log n + \frac12 \log \log n +\omega}{n}$, then the expected number of rooted copies of $M'$ is
\begin{align*}
\binom{n}{3}3p(1-p)^{2(n-2)} & \le (1+o(1)) \frac{n^3p}{2} \exp(-2pn)\\
& = (1+o(1))\frac{n^2\log n}{2} \exp \left( -2\log n - \log\log n -2\omega \right)\\
& \le \exp (-2\omega)=o(1)
\end{align*}
and so by Markov's inequality, whp
$p_M^* \le \frac{\log n + \frac12 \log \log n +\omega}{n}$,
and by Lemma~\ref{lem:nomoreminimal} we have $p_M^*=p_M$ whp, completing the argument.
\end{proof}

\section{The subcritical case}\label{sec:subcritical}
In this section we will prove the first statements of Theorems~\ref{thm:binomialthreshold} and~\ref{thm:hittingtime}.
Unlike many other similar results on connectivity in random graphs or hypergraphs,
the subcritical case is far from trivial. The reason for this is that \connectivity\ 
is not a monotone property, and therefore it is not enough to prove that $\cH_p$
is not \connected\ at time $p=\frac{\log n+\frac12 \log \log n-\omega}{n}$. Rather
we have to prove that whp $\cH_p$ is not \connected\ at every $p$ up to and including this one.

We begin by proving that whp $\cH_p$ becomes topologically connected
at the moment when the last isolated vertex disappears, and that this
occurs at around $p=\frac{2\log n}{n^2}$. The proof is a simple
adaptation of the corresponding result for graphs, and indeed follows
as a special case of previously proved hypergraph results in~\cite{Poole14}
and~\cite{CooleyKangKoch15b}, but we reprove it here for completeness.

\begin{lem}\label{lem:isolated}
Let $\delta>0$ be constant. With high probability $\cH_p$ contains isolated vertices if $p\le (2-\delta) \frac{\log n}{n^2}$. In particular, $\cH_p$ is not topologically connected and therefore also not \connected.

On the other hand, with high probability $\cH_p$ contains no isolated vertices if $p\ge (2+\delta)\frac{\log n}{n^2}$.
\end{lem}

\begin{proof} We note that the presence of isolated vertices \emph{is} a monotone decreasing property, therefore it suffices to prove each statement for the upper or lower bound on $p$ respectively.
For the first statement, let $p=(2-\delta) \frac{\log n}{n^2}$.
The probability that a vertex is isolated is
$$(1-p)^{\binom{n-1}{2}}\ge e^{-pn^2/2-p^2n^2/2}\ge(1+o(1)) n^{-(1-\delta/2)}.$$
Thus if $X$ is the number of isolated vertices, we have $\EE(X)\ge (1+o(1))n^{\delta/2}$.
Furthermore, the probability that two distinct vertices are isolated is
$$(1-p)^{2\binom{n-1}{2}-(n-2)}\le e^{-p(n-2)^2} = n^{-(2-\delta)}e^{O(np)} \le (1+o(1))n^{-(2-\delta)}$$
and so
$$\EE(X^2) \le n(n-1)(1+o(1))n^{-(2-\delta)} + \EE(X) = (1+o(1))\EE(X)^2.$$
It follows that $\var(X) = o(\EE(X)^2)$ and
therefore by Chebyshev's inequality, whp there are isolated vertices as required.

The second statement simply follows from a first moment calculation: For $p=(2+\delta)\frac{\log n}{n^2}$ we have
$$\EE(X) = n(1-p)^{\binom{n-1}{2}} \le ne^{-pn^2/2 + pn/2} \le (1+o(1))n^{-\delta/2} = o(1)$$
so by Markov's inequality, whp there are no isolated vertices.
\end{proof}

We call the components of the shadow graph \emph{topological components},
and say that such a component is \emph{trivial} if it consists of just
one isolated vertex.

\begin{lem}\label{lem:trivialtopological}
Let $\delta>0$ be constant. With high probability, for all $\frac{\log n}{n^2}\le p \le (2+\delta)\frac{\log n}{n^2}$, there is exactly one non-trivial topological component in $\cH_p$.
\end{lem}

Let us note that the constant (i.e.\ $1$) in the lower bound for $p$ is not
the optimal constant (which would in fact be $2/3$),
but this result will be strong enough for our purposes and choosing this
larger constant will make the proof significantly easier.

\begin{proof}
We first show that for $p=\frac{\log n}{n^2}$, whp there are no topological components of size $k$ for $3\le k \le n/4$ (note that a topological component of size $2$ is not possible).

So consider the expected number of topological components of size $k$. There are at most $\binom{n}{k}$ ways of choosing the vertices, and there must be at least $\tfrac{k-1}{2}\ge \tfrac{k}{3}$ faces, with at most $\binom{k}{3}^{k/3}$ ways of choosing these $k/3$ faces. Finally, any triple containing at least one vertex from these $k$ and one vertex from the remaining $n-k$ cannot be a face. Thus there must be at least $k(n-k)(n-2)/2$ non-faces. Thus the expected number of topological components of size $k$ is at most
\begin{align*}
\binom{n}{k}(1-p)^{k(n-k)(n-2)/2}p^{k/3}\binom{k}{3}^{k/3} & \stackrel{\phantom{\left(k\le\frac{n}{4}\right)}}{\le} \left(\left(\frac{en}{k}\right)^3(1-p)^{3(n-k)(n-2)/2}p\frac{k^3}{6}\right)^{k/3}\\
& \stackrel{\phantom{\left(k\le\frac{n}{4}\right)}}{\le} \left(\Theta \left( n(\log n) e^{-3p(n-k)(n-2)/2} \right) \right)^{k/3}\\
& \stackrel{\left(k\le\frac{n}{4}\right)}{\le} \left(\Theta \left( n(\log n) e^{-\frac{3}{2}\log n \cdot \frac{3}{4}(1+o(1))} \right) \right)^{k/3}\\
& \stackrel{\phantom{\left(k\le\frac{n}{4}\right)}}{\le} \left(\Theta \left( n(\log n) n^{-10/9} \right) \right)^{k/3}\\
& \stackrel{\phantom{\left(k\le\frac{n}{4}\right)}}{\le} n^{-k/30}.
\end{align*}
Thus the probability that there is any topological component of size between $3$ and $n/4$ is at most
\[
\sum_{k=3}^{n/4} n^{-k/30} \le n^{-1/11}=o(1).
\]
Thus at time $p=\frac{\log n}{n^2}$, whp we have at most four non-trivial
topological components, each of size at least $n/4$. We now show that
within the time interval
$\frac{\log n}{n^2} \le p \le (1+\delta)\frac{\log n}{n^2}$, these
components will merge together whp. The probability that two such
components do not merge once we add the additional probability
of (at least) $p'=\delta\cdot\frac{\log n}{n^2}$ is at most
\[
(1-p')^{\frac{n^2}{16}\cdot \frac{n-2}{2}} \le \exp \left(-p'n^3/33\right) =\exp \left( -\frac{\delta}{33}n\log n\right) = o(1)
\]
and since we have a bounded number of topological components, whp they all merge together.

Finally, we need to show that from time $p=\frac{\log n}{n^2}$ onwards we don't create
any more non-trivial topological components. We observe that if any such component is created
in the process, then at the time that it first becomes non-trivial it will have size 3. Therefore it is enough to
show that whp, from time $p=\frac{\log n}{n^2}$ onwards
no components of size $3$ are created from previously isolated vertices.

Let us first consider the number of isolated vertices at time $p=\frac{\log n}{n^2}$. The expected number is
\[
n(1-p)^{\binom{n-1}{2}} \le n\exp\left(-p\frac{(n-2)^2}{2}\right) = n\exp \left(-\left(\tfrac12+o(1)\right)\log n\right)\le n^{\frac12+o(1)}.
\]
Thus by Markov's inequality, whp there are at most $n^{3/5}$ isolated vertices.

Now conditioned on having at most $n^{3/5}$ isolated vertices at time $p=\frac{\log n}{n^2}$,
let us denote the set of isolated vertices by $X$. To create a new topological component of size $3$, a face
containing three vertices of $X$ would have to be born before any of the faces containing one of these
vertices and vertices not from $X$. For any one vertex of $x\in X$, the number of faces containing $x$ and
two other vertices of $X$ is at most $n^{6/5}$, and so the probability of one of these faces being the next
to be born is $O(n^{6/5}/n^2)=O(n^{-4/5})$. Taking a union bound over all $O(n^{3/5})$ vertices of $X$,
the probability of creating any topological component of size $3$ is at most $O(n^{-1/5}) = o(1)$.
\end{proof}

From Lemmas~\ref{lem:isolated} and~\ref{lem:trivialtopological} we immediately deduce the value
of the birth time $p_T$ of the face that makes the complex topologically connected.

\begin{cor}\label{cor:topologicallyconn}
  With high probability, $p_T=(2+o(1))\frac{\log n}{n}$.
\end{cor}

Recall that we split the range from $p_T$ to $p_M$, the moment when the last copy of $M$ disappears, into four intervals $I_1,\ldots,I_4$, and aim to show that whp for each $1\le i \le 4$, there is a copy $M_i$ of $M$ which remains in place throughout the interval $I_i$.
Also recall that $\eps$ is a fixed constant with $0<\eps<\frac1{10}$.

\begin{lem}\label{lem:firstobstruction}
With high probability, at the moment the shadow graph becomes connected, the face which we just added forms a copy $M_1$ of $M$, and remains an $M$ until $p_1=n^{-1-\eps}$.
\end{lem}

\begin{proof}
By Lemmas~\ref{lem:isolated} and~\ref{lem:trivialtopological}, whp the shadow graph becomes connected when the last isolated vertex disappears, and this occurs at time about $\frac{2\log n}{n^2}$.

We claim that whp, at the moment the shadow graph becomes connected, we only connected one isolated vertex rather than $2$. For conditioned on the moment when the number of isolated vertices becomes at most $2$, if there is only one left, then the claim follows, but if there are two, then the probability that the next face containing one of them contains both is $\frac{n-2}{2\binom{n-2}{2}}=O(1/n)$.

Therefore whp we had just one isolated vertex $x$ before adding the face $e=xyz$. Then $yz$ is distributed uniformly at random in $V-x$, and the probability that the edge $yz$ was already in a face is at most $np =O\left(\frac{\log n}{n}\right)$ (conditioned on the high probability event that $p\le \frac{(2+\delta) \log n}{n^2}$).

However, $y,z$ were in a connected component before we added the face $e$,
therefore there is a $yz$-path in the shadow graph not using this face. Thus $xyz$ forms a copy $M_1$ of $M$.

It remains to show that $M_1$ remains an $M$ for a long time. This is certainly the case for as long as $xy$ and $yz$ are contained in no further faces. Sprinkling a further probability of (at most) $n^{-1-\eps}$ means that the probability of finding another face containing one of these edges is at most $2n\cdot n^{-1-\eps}=o(1)$, as required.
\end{proof}

\begin{lem}\label{lem:obstructionintervals}
With high probability there are copies $M_2,M_3,M_4$ of $M$ which remain in place for the following time ranges
\begin{itemize}
\item $M_2$ for $n^{-1-\eps} \le p \le \frac{\log n}{10n}$;
\item $M_3$ for $\frac{\log n}{10n} \le p \le \frac{\log n}{n}$;
\item $M_4$ for $\frac{\log n}{n}\le p \le p_M$ (i.e. $M_4$ is the last copy of $M$ to disappear).
\end{itemize}
\end{lem}

\begin{proof}
By Lemma~\ref{lem:secondmoment}, at time $p_2=\frac{\log n}{10n}$ there are at least
\begin{align*}
 \frac{n^3p}{3} \exp(-2pn) & = \Theta \left(n^2\log n \exp \left(-\frac{\log n}{5}\right)\right)\\
& \ge n^{9/5}
\end{align*}
copies of $M'$. By Lemma~\ref{lem:denseshadow}, whp for $p\ge n^{-1-\eps}$, such a triple will always have been a copy of $M$ provided the corresponding face exists. We therefore need to show that whp, at least one of these faces already existed at time $p_1=n^{-1-\eps}$.

To do this, observe that given that these faces exist at time $p_2$, their
birth times are uniformly distributed in $[0,\frac{\log n}{10n}]$. The probability that any fixed
such face existed at time $p_1 \ge \frac{1}{n^{2\eps}}\frac{\log n}{10n}$ is at least
$1-n^{-2\eps}$. Thus, the probability that none of them existed at time $p_1$ is at most
\[
(1-n^{-2\eps})^{n^{9/5}} \stackrel{\left(\eps<\frac1{10}\right)}{\le} e^{-n^{8/5}}=o(1)
\]
as required.

An essentially identical argument also shows that whp there is a minimal obstruction
throughout $p\in [\frac{\log n}{10n},\frac{\log n}{n}]$
(since at time $p_3= \frac{\log n}{n}$ whp we have a growing number of copies of $M$ by Lemma~\ref{lem:secondmoment}),
and that the final minimal obstruction to disappear, at time $p_M = \frac{\log n +\frac12 \log \log n +O(1)}{n}$ already existed at time $p_3=(1-o(1))p_M$.
\end{proof}

Together, Lemmas~\ref{lem:firstobstruction} and~\ref{lem:obstructionintervals} prove the subcritical case of Theorem~\ref{thm:hittingtime}, and together with Corollary~\ref{cor:lastobstruction}, this proves the subcritical case of Theorem~\ref{thm:binomialthreshold}.

\section{The supercritical case}\label{sec:supercritical}

In this section we prove the supercritical cases of the two main theorems. We need the following definition.

\begin{definition}\label{def:superconnected}
  An edge set in a $2$-complex $\cH$ is called \emph{super-connected} if it cannot be
  partitioned into two non-empty sets such that every face of $\cH$ has edges in
  at most one of the two sets.
\end{definition}

Note that an alternative and equivalent definition of super-connectedness
comes from considering an auxiliary graph $G$ whose vertices are the edges
of $\cH$ and with two such vertices connected by an edge if the corresponding
edges lie in a common face of $\cH$. Then a super-connected set of edges in
$\cH$ corresponds to a set of vertices in $G$ which induces a connected
subgraph.

\begin{lem}\label{lem:minimalbad}
  Let $\cH$ be an arbitrary $2$-complex and let $F$ be an edge set
  in $\cH$ that is the support of a bad 0-1
  function and smallest possible with that property. Then
  \begin{enumerate}
  \item\label{minimalbad:degree}
    every vertex of $\cH$ has degree less than $\frac{n}{2}$ in $F$ and
  \item\label{minimalbad:superconnected}
    $F$ is super-connected.
  \end{enumerate}
\end{lem}

Let us note that a similar observation, but with only connected in
place of super-connected, was in~\cite{LinialMeshulam06}.

\begin{proof}
  Suppose, for a contradiction, that $F$ does not
  satisfy~\ref{minimalbad:degree} and let $v$ be a vertex of degree
  $d_F(v)\ge\frac{n}{2}$. Let $E_v$ be the set of edges of $\cH$ at $v$
  and let $F'$ be the symmetric difference of $F$ and $E_v$, i.e.\ an
  edge of $\cH$ is in $F'$ if and only if it either is in $F$ and not
  incident with $v$ or is incident with $v$ and not in $F$. By
  construction, since $F$ is the support of a bad 0-1 function, so is
  $F'$. But
  \begin{equation*}
    |F'| = |F| - d_F(v) + \big(n-1-d_F(v)\big) < |F|,
  \end{equation*}
  contradicting the minimality of $F$. This proves~\ref{minimalbad:degree}.

  Now suppose that $F$ is not super-connected and let $(F_1,F_2)$ be
  a partition witnessing this fact. By the minimality of $F$, both of
  the functions $f_1,f_2$ with support $F_1,F_2$ are not bad, i.e.\
  each $F_i$ either is odd on the boundary of some face or every
  cycle in $\cH$ meets $F_i$ in even number of edges. Since $f$ is bad,
  there is a cycle $C$ in $\cH$ that meets $F$ in an odd number of
  edges. Without loss of generality, $C$ also meets $F_1$ in an odd
  number of edges. Since $f_1$ is not bad, there is a face $\sigma$
  whose boundary meets $F_1$ in an odd number of edges. By the choice
  of $(F_1,F_2)$, the boundary of $\sigma$ avoids $F_2$ and thus meets
  $F$ in an odd number of edges, a contradiction to the fact that $f$
  is bad. This proves~\ref{minimalbad:superconnected}.
\end{proof}

\begin{lem}\label{lem:smallsupport}
  For $p=(1+o(1))\frac{\log n}{n}$, whp there is no bad
  0-1 function on the edges of $\cH_p$ with super-connected support of
  size $3\le k\le \log n$.
\end{lem}

\begin{proof}
  For given $k\ge 3$ we calculate the number of ways of choosing a support of size $k$.
  Since the support must be connected, it covers at most $k+1$ vertices and the number
  of ways of choosing it is at most
$$
\binom{n}{k+1}\binom{\binom{k+1}{2}}{k} \le \left(\frac{en}{k+1}\right)^{k+1}\left(\frac{e(k+1)}{2}\right)^k \le (10n)^{k+1}
$$
  (This calculation was in~\cite{LinialMeshulam06}.)
  There are at most $k^{k-2}$ ways in which the support can be
  super-connected by $k-1$ faces with at least two edges in the support and
  each such face is present with probability $p$. In total, the
  probability that the chosen support is actually super-connected is
  at most $k^{k-2}p^{k-1}$. Each of the (at least) $k(n-k-1)$ triples with
  two vertices forming an edge in the support and the third being
  elsewhere is not allowed to be in $\cH_p$. Therefore, the
  probability that a bad 0-1 function as in the claim exists is
  \begin{align*}
    P_1 &\le \sum_{k=3}^{\log n}(10n)^{k+1}k^{k-2}p^{k-1}(1-p)^{k(n-k-1)}\\
    &\le  \sum_{k=3}^{\log n}\frac{10n}{k^2 p} \left(10nkpe^{-(1+o(1))\frac{\log n}{n}(n-k-1)}\right)^k\\
    &\le \sum_{k=3}^{\log n}n^2 \left(11k\frac{\log n}{n^{1+o(1)}}\right)^k\\
    &\le \sum_{k=3}^{\log n} n^2 \left(\frac{(\log n)^2}{n^{3/4}}\right)^k\\
    &\le (\log n) n^{-1/5} = o\left(1\right).\qedhere
  \end{align*}
\end{proof}

We note that a similar calculation for $P_1$ also works for $k$ up to $n^{1-\eps}$. However, we do not need this since we will cover this range with a different argument which we use for all large $k$.

For the case $k=2$, the above calculation is not strong enough, but in this case we can simply apply Corollary~\ref{cor:lastobstruction}.
Furthermore, bad functions with support of size $k=1$ are not possible, because every 0-1 function with support of size one has an odd face.

For very large $k$ (larger than $\log n$), the bound above on the number of super-connected supports becomes very weak,
so we will need a different way of counting them (we use a breadth-first search). In particular, once $k$ becomes larger
than linear, the bound $k(n-k-1)$ on the number of non-faces becomes useless and we also need a better way of counting
these. For this, we quote the following result of Linial and Meshulam.

\begin{prop}[{\cite[Proposition 2.1]{LinialMeshulam06}}]\label{prop:badtriples}
Let $G$ be any graph on $n$ vertices whose edges are the support of some smallest bad 0-1 function for a $3$-uniform hypergraph. (So in particular, $G$ has exactly one non-trivial component, maximum degree at most $n/2$ etc.) Let $B(G)$ be the number of bad triples, i.e.\ triples containing an odd number of edges of $G$.

Then $B(G)\ge \frac{1}{120}|E(G)|n$.
\end{prop}
For convenience we define $c:=\frac{1}{120}$.
With the help of this proposition we can prove the range when $k$ is large.
In this case the error probabilities are small enough that we may rule out
a bad function not just for \emph{some} $p$, but for \emph{all} $p$ large enough.

\begin{lem}\label{lem:largesupport}
  Whp for all $p\ge (1+o(1)) \frac{\log n}{n}$ there is no smallest bad
  0-1 function on the edges of $\cH_p$ with support of
  size $k\ge \log n$.
\end{lem}

\begin{proof}
Recall that the support of a smallest bad 0-1 function must be super-connected (Lemma~\ref{lem:minimalbad}\ref{minimalbad:superconnected}),
and therefore we can discover it from an edge via a breadth-first search.
More precisely, start from any edge of the support and query all triples
containing it. Any triple which forms a face must have exactly one further
edge of the support contained in it (otherwise it would be odd). From all
further support edges found in this way, according to some arbitrary but
pre-determined order, we continue the process (though querying only triples
containing edges not yet known to be in the support). By the
super-connectivity, we must find the whole support in this way.

Let us bound the number of components of size $k$ which can be found by this process. From each edge we may query up to $n$ triples, and
suppose that from the $i$-th edge we find $b_i$ faces in the BFS. The number of possible ways this can occur
is at most $\binom{n}{b_i} 2^{b_i}$ (we choose which $b_i$ faces are present, and for each face exactly
one of the two further edges within it must be in the support). Thus conditioned on the sequence $(b_i)$, the
total possible number of supports of size $k$ is at most
$$
\binom{n}{2}\prod_{i=1}^{k}\binom{n}{b_i} 2^{b_i} \le \binom{n}{2}\frac{(2n)^{k-1}}{\prod_{i=1}^k b_i!}
$$
where the inequality holds because the $b_i$ must sum to $k-1$. Furthermore, by Proposition~\ref{prop:badtriples} the probability that one fixed support exists and has no odd face is at most
$$
p^{k-1}(1-p)^{ckn} \le \left(p(1-p)^{cn}\right)^{k-1}.
$$
By differentiating the expression $p(1-p)^{cn}$, we can see that
this expression has its maximum at $p=\frac{1}{1+cn} \ll \frac{\log n}{n}$.
Therefore in the following calculations, we may substitute in the lower bound
of $(1+o(1))\frac{\log n}{n}$ for $p$.
Thus the probability $p^*$ that some such support exists and has no odd face satisfies
\begin{align*}
p^*\prod_{i=1}^k b_i! & \le \binom{n}{2}\left(2np(1-p)^{cn}\right)^{k-1}\\
& \le \binom{n}{2}\left( 3(\log n) e^{-(c+o(1))\log n}\right)^{k-1}\\
& \le \left(3(\log n) n^{-2c/3+2/(k-1)} \right)^{k-1}\\
& \le n^{-ck/2}.
\end{align*}
Here for the last inequality we used the fact that $k\ge \log n$.

However, we still need to sum over all possible sequences $b_i$. We now make a case distinction based on the number of $b_i$ which are large. Let
$$
\ell := |\{i:b_i\ge n^{c/4}\}|.
$$
For fixed $\ell$, we very crudely bound the number of possible sequences $b_i$ by
$$
k^\ell n^\ell (n^{c/4})^{k-\ell}.
$$
(We choose which $\ell$ positions have $b_i\ge n^{c/4}$, for each of these we choose a $b_i$ at most $n$, for all others we choose a $b_i$ at most $n^{c/4}$.)

On the other hand, we have
$$
\prod_{i=1}^k b_i! \ge ((n^{c/4})!)^\ell \ge (n^{c/4})^{\ell n^{c/5}} \ge n^{\ell n^{c/6}}.
$$
Thus
\begin{align*}
\sum_{b_1,\ldots,b_k}\frac{1}{\prod_{i=1}^k b_i!} & \le \sum_{\ell=0}^k \frac{k^\ell n^\ell (n^{c/4})^{k-\ell}}{n^{\ell n^{c/6}}}\\
& = n^{ck/4} \sum_{\ell=0}^k \left(\frac{kn}{n^{c/4}\cdot n^{n^{c/6}}}\right)^\ell \\
& \le n^{ck/4} \cdot (k+1).
\end{align*}

Combining this with our previous bounds, the probability that there exists a bad support of size $k$ is at most
$$
n^{-ck/2}n^{ck/4} \cdot (k+1) \le n^{-ck/5}.
$$
Summing over all $k\ge \log n$, the probability that any such bad support exists is
at most $n^{-\frac{c}{6}\log n}$.
This bound is valid for any single $p\ge (1+o(1))\frac{\log n}{n}$,
and taking a union bound over all $O(n^3)$ birth times in this range,
we conclude that the probability that any bad minimal support of size
at least $\log n$ ever appears in this range is at most
$$
n^{-\frac{c}{7}\log n} =o(1)
$$
as required.
\end{proof}

We note here that there is nothing particularly special about
the bound $k\ge\log n$. In fact, an identical argument
works for $k$ larger than some absolute constant which is related to $c$.

Together with the fact that there are no bad functions with support
of size one, Corollary~\ref{cor:lastobstruction} and
Lemmas~\ref{lem:smallsupport} and~\ref{lem:largesupport}  show
that at time $p=\frac{\log n + \frac12 \log \log n + \omega}{n}$, whp $\cH_p$
is \connected. However, we still need to prove that whp it does not become
disconnected again. We aim to do this by showing that a new obstruction
cannot appear suddenly at a large size -- rather, we must first see a
copy of $M$, which by Corollary~\ref{cor:lastobstruction} we already
know does not happen whp. 

For supports of size at least $\log n$, this is already implied by
Lemma~\ref{lem:largesupport}. We proved this since the error probability
was small enough that we could take a union bound over all remaining birth times,

Indeed, a similar argument would work for supports of size at least $6$,
but this calculation does not work
for $k\le 5$. Since for small supports we have to be more careful in any case,
we use this argument for all sizes up to $\log n$.

For a $2$-complex $\cH$ and a triple of vertices $T$, let $\cH+T$ denote the complex
obtained from $\cH$ by adding $T$ as a face and (if necessary) all pairs in $T$ as
edges.

\begin{lem}\label{lem:throughM}
Suppose that in a $2$-complex $\cH$ each pair of vertices
is connected by at least $k$ paths of length $2$ in the shadow graph.
Suppose further that $\cH$ is \connected, but for some triple $T$, $\cH+T$ contains
a bad 0-1 function with support of size at most $k$.
Then $M \subset \cH+T$.
\end{lem}

\begin{proof}
Let $f$ be a bad 0-1 function in $\cH+T$ with minimal support.

Suppose first that the support of $f$ contains an edge $e$ outside of $T$.
Let $S$ be the support of $f$ and let $S'$ be a maximal subset of $S$
which contains $e$ and is super-connected in $\cH$. Let $f'$ be a new
0-1 function whose support is exactly $S'$.

Now by the maximality of $S$, any face of $\cH$ meeting $S'$ cannot meet $S\setminus S'$, and
since such a face was even with respect to $f$, it is also
even with respect to $f'$. On the other hand, since $|S'\setminus\{e\}|\le k-1<k$,
the edge $e\in S'$ is contained within a triangle of $\cH$ which contains
no further edges of $S'$, and thus forms an odd cycle with respect to $f'$. But this means
that $\cH$ is not \connected, which is a contradiction.

Thus the support of $f$ is contained within $T$. But then
(by the fact that $T$ must be an even face)
the support consists of exactly $2$ edges and $T$ forms a copy of $M$ in $\cH+T$.
\end{proof}

We can now complete the proof of the supercritical case. For we know
by Lemma~\ref{lem:largesupport} that whp the smallest obstruction is
of size smaller than $\log n$ for any $p\ge \frac{\log n}{n}$.
Furthermore, we know that at time $p=p_M > \frac{\log n}{n}$, $\cH_p$
is \connected\ whp by Lemma~\ref{lem:smallsupport}. Finally, by Lemma~\ref{lem:denseshadow} $\cH_p$ satisfies
the shadow graph condition of Lemma~\ref{lem:throughM}.

Let us condition on these high probability events all occurring.
Now suppose for some $p\ge p_M$, $\cH_p$ is not \connected, and let $p$ be
minimal such that this is the case. Then $p$ is the birth-time of
a face $T$, and $\cH$, the complex just before this face was born,
satisfies all the conditions of Lemma~\ref{lem:throughM}. But then
$\cH_p=\cH+T$ contains a copy of $M$, contradicting the fact that $p\ge p_M$.

\section{Concluding remarks}\label{sec:conclusion}

\subsection{Search processes for hypergraphs}

In Section~\ref{sec:supercritical} we used a search process to allow
us to better count the possible number of super-connected supports.
Such search processes in hypergraphs have been used previously, for
example in~\cite{CooleyKangPerson14} and~\cite{LuPeng14} to determine
the threshold for high-order phase transitions in hypergraphs,
inspired by previous work for graphs in~\cite{KrivelevichSudakov13}.

\subsection{Alternative models}

There are several possible ways of generating a random $2$-complex.
If we start from a random binomial $3$-uniform hypergraph, we must
add some edges to ensure that we have a complex. The model of
Linial and Meshulam and the model we consider in this paper lie at
the two extremes---either adding in all possible edges, or only
adding those edges we really have to. One might also consider what
happens in between, if only \emph{some} of the (not strictly
necessary) edges are added, possibly randomly. However, as far as the
\connectivity\ of the resulting complex is concerned, this is
essentially covered
by the results of~\cite{LinialMeshulam06} and this paper. Indeed, if any
edge is not contained in a face, the complex is not \connected;
otherwise we have the model we considered here.

It is also possible to construct a $2$-complex from a graph rather
than from a $3$-uniform hypergraph by
taking all triangles of the graph as faces. This is a special
case of the \emph{clique complex} which has been studied
for example in~\cite{Meshulam01} and~\cite{Kahle09}.

\subsection{Higher dimension}

A natural question would be to ask whether the results in this paper extend to
higher-dimensional complexes. For a $k$-complex generated from a
$(k+1)$-uniform hypergraph, we could ask whether the $0$-th, $1$st, \ldots, $(k-1)$-th
homology groups all vanish.

For the analogue of the $\cHLM_p$ model of Linial and Meshulam the obvious conjecture
is that the $(k-1)$-th homology group should vanish once every set
of $k$ vertices is contained in a $k$-simplex, which property has a
threshold function of $p=\frac{k\log n}{n}$. This was indeed confirmed
to be the threshold in~\cite{MeshulamWallach08}. The behaviour within
the critical window was subsequently examined in~\cite{KahlePittel16}.

In the analogue of our $\cH_p$ model, we would have to consider the vanishing of
the $j$-th homology group separately for $j=0,\dotsc,k-1$. It is not too
hard to see that the vanishing of the zero-th homology group in the $k$-dimensional model has threshold
$p=\frac{\log n}{\binom{n}{k}}$. For general $j$, we expect the threshold
for the vanishing of the $j$-th homology group to be of order $\frac{\log n}{n^{k-j}}$.

\subsection{Alternative definitions of connectivity}

If $X$ is a simplicial complex, then the vanishing of $H_1(X;\FF_2)$
is just one way of defining ``one-dimensional connectivity''. A stronger
notion would be to ask for the homology group $H_1(X;\ZZ)$ to vanish.
Similar notions of connectivity---for arbitrary dimension and coefficients
in an arbitrary finite group---have been considered in~\cite{MeshulamWallach08}.
The strongest notion of one-dimensional connectivity is to consider
\emph{simple connectivity}, i.e.\ the vanishing of the fundamental group.
For the $\cHLM_p$ model, this was studied in~\cite{BabsonHoffmanKahle11}.
Another possibility is to consider Betti numbers~\cite{Kozlov10}.

For hypergraphs rather than complexes, vertex-connectivity is
by far the most studied definition, and the connectivity threshold
of $\frac{\log n}{\binom{n}{k}}$ in a $(k+1)$-uniform hypergraph
and a corresponding hitting time result
can easily be proved analogously to the graph case. (For $k=2$
this is proved in this paper, albeit not with the sharpest possible threshold,
in Lemmas~\ref{lem:isolated} and~\ref{lem:trivialtopological}).

More generally, for $1\le j \le k$ and a $(k+1)$-uniform hypergraph we may define
a higher-order notion of connectivity on the $j$-sets (the case $j=1$ corresponds
to vertex connectivity). Then the threshold for connectivity is
$\frac{j\log n}{\binom{n}{k+1-j}}$, as proved in~\cite{CooleyKangKoch15b}.

\ 

\bibliographystyle{amsplain}
\bibliography{../References.bib}

\providecommand{\bysame}{\leavevmode\hbox to3em{\hrulefill}\thinspace}
\providecommand{\MR}{\relax\ifhmode\unskip\space\fi MR }
\providecommand{\MRhref}[2]{%
  \href{http://www.ams.org/mathscinet-getitem?mr=#1}{#2}
}
\providecommand{\href}[2]{#2}
\begin{thebibliography}{10}

\bibitem{BabsonHoffmanKahle11}
E.~Babson, C.~Hoffman, and M.~Kahle, \emph{The fundamental group of random
  2-complexes}, J. Amer. Math. Soc. \textbf{24} (2011), no.~1, 1--28.
  \MR{2726597}

\bibitem{BollobasThomason85}
B.~Bollob{\'a}s and A.~Thomason, \emph{Random graphs of small order},
  North-Holland Math. Stud., vol. 118, pp.~47--97, North-Holland, Amsterdam,
  1985. \MR{860586 (87k:05137)}

\bibitem{CooleyKangKoch15b}
O.~{Cooley}, M.~{Kang}, and C.~{Koch}, \emph{{Threshold and hitting time for
  high-order connectivity in random hypergraphs}}, arXiv:1502.07289
  (\emph{Submitted}, 2015).

\bibitem{CooleyKangPerson14}
O.~{Cooley}, M.~{Kang}, and Y.~{Person}, \emph{{Largest components in random
  hypergraphs}}, ArXiv e-prints 1412.6366 (\emph{Manuscript}, 2014).

\bibitem{ErdosRenyi59}
P.~Erd{\H{o}}s and A.~R{\'e}nyi, \emph{On random graphs. {I}}, Publ. Math.
  Debrecen \textbf{6} (1959), 290--297. \MR{0120167 (22 \#10924)}

\bibitem{Kahle09}
M.~Kahle, \emph{Topology of random clique complexes}, Discrete Math.
  \textbf{309} (2009), no.~6, 1658--1671. \MR{2510573}

\bibitem{KahlePittel16}
M.~Kahle and B.~Pittel, \emph{Inside the critical window for cohomology of
  random {$k$}-complexes}, Random Structures \& Algorithms \textbf{48} (2016),
  no.~1, 102--124. \MR{3432573}

\bibitem{Kozlov10}
D.~Kozlov, \emph{The threshold function for vanishing of the top homology group
  of random {$d$}-complexes}, Proc. Amer. Math. Soc. \textbf{138} (2010),
  no.~12, 4517--4527. \MR{2680076}

\bibitem{KrivelevichSudakov13}
M.~Krivelevich and B.~Sudakov, \emph{The phase transition in random graphs: A
  simple proof}, Random Structures \& Algorithms \textbf{43} (2013), no.~2,
  131--138.

\bibitem{LinialMeshulam06}
N.~Linial and R.~Meshulam, \emph{Homological connectivity of random
  2-complexes}, Combinatorica \textbf{26} (2006), no.~4, 475--487. \MR{2260850
  (2007i:55004)}

\bibitem{LuPeng14}
L.~{Lu} and X.~{Peng}, \emph{{High-order Phase Transition in Random
  Hypergrpahs}}, ArXiv e-prints 1409.1174 (\emph{Manuscript}, 2014).

\bibitem{Meshulam01}
R.~Meshulam, \emph{The clique complex and hypergraph matching}, Combinatorica
  \textbf{21} (2001), no.~1, 89--94. \MR{1805715}

\bibitem{MeshulamWallach08}
R.~Meshulam and N.~Wallach, \emph{Homological connectivity of random
  {$k$}-dimensional complexes}, Random Structures \& Algorithms \textbf{34}
  (2009), no.~3, 408--417. \MR{2504405}

\bibitem{Poole14}
D.~{Poole}, \emph{{On the strength of connectedness of a random hypergraph}},
  ArXiv e-prints 1409.1489 (\emph{Manuscript}, 2014).

\end{thebibliography}

\

\end{document}